\documentclass[11pt]{amsart}

\usepackage{mystyle}

\begin{document}

\title{On the existence of reflecting $n$-queens configurations}

\date{}

\author{Tantan Dai}
\address{Georgia Institute of Technology}
\email{tdai44@gatech.edu}

\author{Tom Kelly}
\address{Georgia Institute of Technology}
\email{tom.kelly@gatech.edu}

\thanks{Research supported by the National Science Foundation under Grant No. DMS-2247078.}

\begin{abstract}
    
    In 1967, Klarner proposed a problem concerning the existence of reflecting $n$-queens configurations. The problem considers the feasibility of placing $n$ mutually non-attacking queens on the \emph{reflecting chessboard}, an $n\times n$ chessboard with a $1\times n$ ``reflecting strip'' of squares added along one side of the board. A queen placed on the reflecting chessboard can attack the squares in the same row, column, and diagonal, with the additional feature that its diagonal path can be reflected via the reflecting strip. Klarner noted the equivalence of this problem to a number theory problem proposed by Slater, which asks: for which $n$ is it possible to pair up the integers 1 through $n$ with the integers $n+1$ through $2n$ such that no two of the sums or differences of the $n$ pairs of integers are the same. We prove the existence of reflecting $n$-queens configurations for all sufficiently large $n$, thereby resolving both Slater's and Klarner's questions for all but a finite number of integers. 

\end{abstract}

\maketitle

\section{Introduction}

An \emph{$n$-queens configuration} is a placement of $n$ queens on an $n\times n$ chessboard such that no two queens are in the same row, column, or diagonal. The classical $n$-queens problem asks how many $n$-queens configurations exist for a given $n \times n$ chessboard. The problem can also be considered on a toroidal chessboard, where the diagonals wrap around the board from left to right and from top to bottom. The problem was first proposed in 1848 by German chess composer Max Bezzel and elicited the interest of several prominent mathematicians, including Gauss and P\'olya. Detailed accounts of the historical development of the $n$-queens problem can be found in the survey by Bell and Stevens \cite{BELL20091} and the work by Bowtell and Keevash \cite{bowtell2021nqueens}.

Let $Q(n)$ denote the number of classical $n$-queens configurations, and let $T(n)$ denote the number of toroidal $n$-queens configurations. In 1874, Pauls \cite{pauls1, pauls2} proved that $Q(n)>0$ for every $n\ge 4$. In 1918, P\'olya \cite{polya} showed that $T(n)>0$ if and only if $n\equiv 1$ or $5\pmod 6$. In 1994, Rivin, Vardi, and Zimmerman \cite{rivin1994} conjectured that $\log Q(n) = \Theta (n\log (n))$ and that $\log T(n) = \Theta (n\log (n))$ for $n\equiv 1,5\pmod 6$. In 2017, Luria \cite{luria2017new} showed that $T(n)\le ((1+o(1))ne^{-3})^n$ and that there exists a constant $\alpha > 1.587$ such that $Q(n)\le ((1+o(1))ne^{-\alpha})^n$ for all $n$. Bowtell and Keevash \cite{bowtell2021nqueens} and independently Luria and Simkin \cite{luria2021lower} proved that $Q(n)\ge ((1+o(1))ne^{-3})^n$ for all sufficiently large $n$.  Bowtell and Keevash \cite{bowtell2021nqueens} also proved that $T(n) \geq ((1 - o(1))ne^{-3})^n$ for $n \equiv 1,5 \pmod 6$, thereby giving an asymptotic solution to the toroidal $n$-queens problem. These results completely settled the conjecture by Rivin, Vardi, and Zimmerman \cite{rivin1994}. Furthermore, Simkin \cite{simkin2022number} improved the bounds for the classical $n$-queens problem, showing that there exists a constant $1.94<\alpha < 1.9449$ such that $Q(n) = ((1\pm o(1))ne^{-\alpha})^n$. 

Glock, Munhá Correia, and Sudakov \cite{glock2022nqueens} investigated a natural extension of the $n$-queens problem, known as the $n$-queens completion problem. This problem asks under what conditions a given placement of mutually non-attacking queens can be extended to an $n$-queens configuration. They showed that any partial configuration with at most $n/60$ queens can be completed. Additionally, they provided a partial configuration of roughly $n/4$ queens that cannot be extended into a complete $n$-queens configuration. A key tool in their work is their ``rainbow matching lemma'', which is also crucial to the proof of our main theorem. 

We consider a slight variation of the $n$-queens problem.  Consider an $n\times n$ chessboard with an additional $1\times n$ strip of squares, called the \emph{reflecting strip}, attached to one side of the chessboard. Without loss of generality, we assume the reflecting strip is placed above the $n\times n$ chessboard.  The squares on the reflecting strip are labeled from left to right as $1,2,\ldots, n$. The \emph{$k$th reflecting diagonal} is the union of the two diagonals (one of which may be empty) whose extension intersects with the $k$th square of the reflecting strip. A diagonal whose extension does not intersect the reflecting strip is called a \emph{non-reflecting diagonal}. An $n\times n$ chessboard with the reflecting strip, along with the rows, columns,  reflecting diagonals, and non-reflecting diagonals defined above, is called a \emph{reflecting chessboard}. Two queens attack each other on a reflecting chessboard if they are in the same row, column, reflecting diagonal, or non-reflecting diagonal. A \emph{reflecting $n$-queens configuration} is a placement of $n$ mutually non-attacking queens on the $n \times n$ reflecting chessboard. Figure~\ref{figure:reflecting} provides an illustration of the reflecting chessboard. The figure on the left illustrates the $3$rd reflecting diagonal (in cyan) and the $8$th reflecting diagonal (in pink) on the reflecting $8\times 8$ chessboard. The figure on the right shows two queens that do not attack each other on a classical $8\times 8$ chessboard but attack each other on the reflecting $8\times 8$ chessboard via the $4$th reflecting diagonal.  See also Definition~\ref{def:reflecting-chessboard}.

\begin{figure}
\begin{minipage}{.5\linewidth}
\centering
    \begin{tikzpicture}
     \filldraw[cyan] (0,3) rectangle (0.5,3.5);
     \filldraw[cyan] (1,4) rectangle (0.5,3.5);
     \filldraw[cyan] (2,4) rectangle (1.5,3.5);
     \filldraw[cyan] (2,3) rectangle (2.5,3.5);
     \filldraw[cyan] (2.5,2.5) rectangle (3,3); 
     \filldraw[cyan] (3,2) rectangle (3.5,2.5); 
     \filldraw[cyan] (3.5,1.5) rectangle (4,2);
     \filldraw[pink] (0,0.5) rectangle (0.5,1);
     \filldraw[pink] (0.5,1) rectangle (1,1.5);
     \filldraw[pink] (1,1.5) rectangle (1.5,2);
     \filldraw[pink] (1.5,2) rectangle (2,2.5);
     \filldraw[pink] (2,2.5) rectangle (2.5,3); 
     \filldraw[pink] (2.5,3) rectangle (3,3.5); 
     \filldraw[pink] (3,3.5) rectangle (3.5,4);
      \foreach \i in {1,2,4,5,6,7}
     {\draw (0.5*\i-0.25,4.25) node{\i};} 
     \draw (1.25,4.25) [cyan] node{3};
     \draw (3.75,4.25) [pink] node{8};
     \draw[step=0.5] (0,0) grid (4,4);
     \draw (0,4.5) -- (4,4.5);
     \draw (0,4) -- (0,4.5);
     \draw (0,4) -- (0,4);
     \draw (4,4) -- (4,4.5);
    \end{tikzpicture}
\end{minipage}%
\begin{minipage}{.5\linewidth}
\centering
\begin{tikzpicture}
     \draw[step=0.5] (0,0) grid (4,4);
     \draw (0,4.5) -- (4,4.5);
     \draw (0,4) -- (0,4.5);
     \draw (4,4) -- (4,4.5);
      \filldraw (0.75,3.25) circle (0.15);
     \filldraw (3.75,2.25) circle (0.15);
     \draw (0.75,3.25) -- (1.75,4.25);
     \draw (3.75,2.25) -- (1.75,4.25);
      \foreach \i in {1,2,3,4,5,6,7,8}
     {\draw (0.5*\i-0.25,4.25) node{\i};} 
    \end{tikzpicture}
\end{minipage}
\caption{Illustrations of the reflecting chessboard}
\label{figure:reflecting}
\end{figure}

In this paper, we prove that reflecting $n$-queens configurations exist for all sufficiently large $n$. This question was first proposed by Klarner \cite{Kl67} in 1967 as an alternative interpretation of a number theory problem of Slater \cite{slater1963}, which we introduce in detail in the following section.  The problem also appears in the book of Guy \cite{Gu04} on unsolved problems in number theory and the survey of Bell and Stevens \cite{BELL20091}.

\subsection{A Related Number Theory Problem} \label{sect:number-theory}
For $n \in \mathbb N$, we let $[n] = \{1, \dots, n\}$.
In 1962, Shen and Shen \cite{shen1962} proposed the following research question: for which $n\ge 3$ is it possible to divide the elements in the set $[2n]$ into pairs $(a_i,b_i)$, such that for all $i\in [n]$, the $2n$ sums and differences $b_i\pm a_i$ are distinct. This problem was solved by Huff \cite{Huff1973OnPO} using a number theoretic approach. In 1963, Slater \cite{slater1963} proposed a more restricted version of this problem, which states: for which values of $n$ is it possible to form pairs $(1,a_1), (2,a_2), \ldots, (n, a_n)$, where $\{a_1,\ldots,a_n\}=\{n+1,n+2,\ldots, 2n\}$, such that for all $i\in [n],$ the $2n$ sums and differences $a_i\pm i$ are all distinct. Slater noted that there is no solution to the problem when $n=2,3,$ and $6$, and conjectured that solutions exist for all other $n.$ Klarner \cite{Kl67} extended this line of inquiry by proposing the question on the existence of reflecting $n$-queens configurations. Klarner showed that Slater's problem has a solution for a given $n$ if and only if there exists a reflecting $n$-queens configuration. To see the equivalence, consider an $n\times n$ reflecting chessboard with rows labeled from top to bottom starting at the row below the reflecting strip as $1,2,\ldots,n$, and columns labeled from left to right as $n+1, n+2, \ldots, 2n.$ Two queens placed on row $i$ column $j$ and on row $i'$ column $j'$ attack each other if they are: \begin{itemize}
    \item in the same row, i.e.,\ $i=i'$, 
    \item in the same column, i.e.,\ $j=j'$, 
    \item on the same ``plus-diagonal'', i.e.,\ $i+j =i'+j'$, 
    \item on the same ``minus-diagonal'', i.e.,\ $i-j =i'-j'$, or
    \item on the same reflecting diagonal and not on the same ``plus-diagonal'' or ``minus-diagonal'', i.e.,\ $i+j =j'-i'$ or $j-i =i'+j'$.
\end{itemize} 
A solution to Slater's version of the problem avoids all these constraints, thus yielding a reflecting $n$-queens configuration, and vice versa. Therefore, there is a one-to-one correspondence between the solutions to Slater's problem for $n$ and reflecting $n$-queens configurations.

In \cite{Kl67}, Klarner showed that reflecting $n$-queens configurations exist for $n=4,5,7$, and $8$. Subsequently, in \cite{sebastian1969}, Sebastian extended this result for $n=9,\ldots, 27.$ Our goal is to establish the existence of reflecting $n$-queens configurations for all sufficiently large $n$. Before presenting the proof, we formulate the problem mathematically and introduce some necessary definitions.

\subsection{Algebraic Formulation}
We formulate the $n\times n$ chessboard as $[n]\times [n]$. We label the rows from top to bottom with $1,2,\ldots,n$, and the columns from left to right with $1,2,\ldots,n$. In the following definition, we define necessary terminologies regarding the chessboard mathematically. 

\begin{definition}\label{def:reflecting-chessboard} Consider a chessboard $[n]\times [n]$.
    \begin{itemize}
        \item  For $i\in [n]$, define \emph{row} $i$ to be $R_i = \{(i,j): j\in[n]\}$. Let $\cR = \{R_i:i\in[n]\}$ denote the set of rows. 
        \item For $j\in [n]$, define \emph{column} $j$ to be  $C_j = \{(i,j): i\in[n]\}$. Let $\cC = \{C_j:j\in[n]\}$ denote the set of columns. 
        \item For $k\in \{-n,\ldots, 0, \ldots, n\}$, define the \emph{$k$th plus-diagonal} to be  $D_k^+ = \{(i,j)\in [n]\times [n]: i+j-(n+1)=k\}$, and the \emph{$k$th minus-diagonal} to be  $D_k^- = \{(i,j)\in [n]\times [n]: i-j=k\}$. Let $\cD = \{D_k^+:k\in\{-(n-1),\ldots, 0,\ldots, n-1\}\} \cup \{D_k^-:k\in\{-(n-1),\ldots, 0,\ldots, n-1\}\}$ denote the set of non-empty diagonals.
        \item A \emph{reflecting diagonal} is defined as $RD_\ell = D_{\ell-(n+1)}^+\cup D_{-\ell}^-$ for $\ell\in [n]$. A \emph{non-reflecting diagonal} is a diagonal in $\cD$ that is not part of a reflecting diagonal. The set of \emph{reflecting diagonals} is denoted by $\mathcal{RD}=\{RD_\ell:\ell\in [n]\}$, and the set of \emph{non-reflecting diagonals} is denoted by  $\mathcal{ND}=\{D_k^+: k\in\{0,\ldots, n-1\}\} \cup \{D_k^-: k\in\{0,\ldots, n-1\}\}$.
        \item A \emph{line} is defined to be an element in the set $\cL=\cR\cup \cC\cup \mathcal{RD} \cup \mathcal{ND}$.
    \end{itemize}
   Although each of these definitions technically depends on $n$, there will be no ambiguity.
\end{definition}
Note that $RD_\ell$ is the union of the two diagonals (one of which may be empty) whose extensions intersect the reflecting strip at the $\ell$th slot.  

Recall that a reflecting $n$-queens configuration is a placement of $n$ queens on an $n \times n$  reflecting chessboard such that no two queens are contained in the same line. We prove the existence of reflecting $n$-queens configurations for all large enough $n$, thereby resolving both Slater's and Klarner's questions for all but finitely many $n$.

\begin{theorem}\label{thm:main-thm}
 A reflecting n-queens configuration exists for all sufficiently large n.
\end{theorem}

In Section~\ref{sect:rainbow-matching}, we discuss the connection between the reflecting $n$-queens problem and the rainbow matching problem, highlighting how insights from the rainbow matching problem can help us with the understanding of the reflecting $n$-queens problem. Section~\ref{sect:proof} provides the necessary tools for proving the main theorem and presents the proof of the main theorem. 

Note that our proof will only work when $n$ is sufficiently large. It remains open to show that  reflecting $n$-queens configurations exist for all $n$. Moreover, we could ask the question about how many possible reflecting $n$-queens configurations there are for any integer $n$. Further variations of the $n$-queens problem can be found in the survey on the $n$-queens problems by Bell and Stevens \cite{BELL20091}. 

\section{The Rainbow Matching Lemma}\label{sect:rainbow-matching}

The proof of our main theorem is inspired by the work of Glock, Munhá Correia, and Sudakov \cite{glock2022nqueens} on the $n$-queens completion problem, which asks under which condition a given partial configuration can be extended to an $n$-queens configuration. To answer this question, Glock, Munhá Correia, and Sudakov presented the ``rainbow matching lemma'' which allowed them to find perfect rainbow matchings in certain graphs. A \emph{rainbow matching} in an edge-colored graph is a matching in which all edges have distinct colors, and a \emph{perfect matching} in a graph is a matching that saturates every vertex. We apply a generalized version of the ``rainbow matching lemma'' to prove our main theorem.  Historical results and recent developments on rainbow matching problems are discussed in \cite{glock2022nqueens}. 

We can translate the problem of the existence of a reflecting $n$-queens configuration into a rainbow matching problem. Given an $n\times n$ reflecting chessboard, we construct a complete bipartite graph $K_{n,n}$ with one part $\cR=\bigcup_{i=1}^n\{R_i\}$ representing the $n$ rows of the  chessboard and the other part $\cC=\bigcup_{i=1}^n\{C_i\}$ representing the $n$ columns. An edge $\{R_i,C_j\}$ corresponds to square $(i,j)$ on the chessboard. A matching in the graph corresponds to a placement of queens in the reflecting chessboard such that no two queens are in the same row or column. We view the diagonals as colors and assign to each edge $\{R_i, C_j\}$ two colors corresponding to the two diagonals containing square $(i,j)$, as follows:
\begin{itemize}
    \item If $0\le i+j-(n+1)\le n$, assign $D_{i+j-(n+1)}^+$ to $\{R_i,C_j\}$. Otherwise, assign $RD_{i+j}$ to $\{R_i,C_j\}$.
    \item If $0\le i-j\le n$, assign $D_{i-j}^-$ to $\{R_i,C_j\}$.
     Otherwise, assign $RD_{j-i}$ to $\{R_i,C_j\}$.
\end{itemize} 
A reflecting $n$-queens configuration corresponds to a perfect rainbow matching, a matching that saturates all the vertices in the graph and in which the color sets of the edges in the matching are pairwise disjoint. 
Similarly, the number theory problem proposed by Shen and Shen \cite{shen1962} discussed in Section~\ref{sect:number-theory} can be formulated as determining whether a particular coloring of $K_{2n}$ has a perfect rainbow matching.
To understand the existence of a perfect rainbow matching in our situation, we need the generalized version of the rainbow matching lemma in \cite{glock2022nqueens}. Before stating the generalized rainbow matching lemma, we need the following definitions. 

\begin{definition} 
    Let $G$ be a graph. A \emph{t-fold} edge-coloring of $G$ is an assignment of sets of $t$ colors to the edges of $G$. This coloring is called \emph{b-bounded} if, for every vertex $v$, each color appears in at most $b$ edges incident to $v$, and any pair of colors appears on at most $b$ edges together. In particular, a $t$-fold coloring is \emph{proper} if the edges at each vertex have pairwise disjoint color sets, and is \emph{linear} if for every pair of colors, there is at most one edge that contains both colors. A subgraph $H$ of $G$ is \emph{rainbow} if for any pair of edges of $H$, their color sets are disjoint. The \emph{degree} of a color is the number of edges whose color set contains that color. 
    
\end{definition}

In the original version of the ``rainbow matching lemma'', Glock, Munhá Correia, and Sudakov \cite[Lemma 2.2]{glock2022nqueens} considered proper and linear $t$-fold coloring of graphs. They remarked that the proper and linear conditions can be weakened to the $t$-fold coloring being $b$-bounded. This leads us to the following generalized version of the rainbow matching lemma.

\begin{lemma}[Glock, Munhá Correia, and Sudakov \cite{glock2022nqueens}]\label{lemma:rainbow}
   For all $\alpha>0$ and $b,t\in \mathbb{N}$, there exists $\varepsilon>0$ such that the following holds for all sufficiently large $n$. Let $G$ be a bipartite graph with parts $A, B$ of size $n$ with a $b$-bounded $t$-fold edge-coloring. If there exists some $d$ such that 
   \begin{enumerate}
    \item every vertex has degree $(1\pm \varepsilon)d$,
    \item every color has degree at most $(1-\alpha)d$, and 
    \item there are at least $\alpha n^2$ edges between any two sets $A'\subseteq A$ and $B'\subseteq B$ of size at least $(1-\alpha) d$,
   \end{enumerate}
   then $G$ has a perfect rainbow matching. 
\end{lemma}

Recall that our goal is to find perfect rainbow matchings in complete bipartite graphs. The lemma tells us that if we have a bipartite graph such that all the vertices have roughly the same degrees, the degree of each color is a bit smaller than that of the vertices, and there are many edges between any sufficiently large sets of vertices, then we can find a perfect rainbow matching. Note that the first and the third conditions are trivially satisfied by the complete bipartite graph $K_{n,n}$.  However, the reflecting queens coloring does not satisfy the second condition as the main diagonals $D_0^+$ and $D_0^-$ have size $n$, and thus the two corresponding colors have degree $n$. The reflecting diagonals also have size $n - 1$.  Nevertheless, on average, the colors have degree $3n/4.$ Thus, our goal is to find a subgraph $G$ of $K_{n,n}$, such that the degree of each color is significantly smaller than that of the vertices. Now, to prove the existence of an $n$-queens configuration, it suffices to show that the subgraph $G$ has a perfect rainbow matching. We will find such a subgraph $G$ in the following section and complete the proof of the main theorem.

\section{Proof of Theorem~\ref{thm:main-thm}}\label{sect:proof}

In this section, we prove Theorem \ref{thm:main-thm}. As mentioned in the previous section, our goal is to find a subgraph of $K_{n,n}$ with a $2$-fold edge-coloring, where edges represent the squares on the reflecting chessboard and the colors of the edges represent the diagonals in which the corresponding square is contained, to which we can apply Lemma~\ref{lemma:rainbow}. In particular, we want to find a subgraph such that the degree of the colors is significantly smaller than the degree of the vertices. Equivalently, we want to find a subset $S$ of the reflecting chessboard such that the diagonals contain significantly fewer squares than the rows and the columns. We prove the existence of such a subset $S$ in the following lemma. 

\begin{lemma}\label{lemma:subset}
    For all sufficiently large $n$, there exists a subset $S\subseteq [n]\times [n]$ such that 
    \begin{enumerate}
        \item each row and each column has $(1\pm n^{-1/4})5n/6$ squares in $S$,
        \item each non-reflecting diagonal and each reflecting diagonal has at most $ 119n/144$ squares in $S$, and 
        \item for every $A, B \subseteq [n]$ satisfying $|A|, |B| \geq 119n/144$, we have $|S \cap (A \times B)| \geq n^2 / 120$.
    \end{enumerate}
\end{lemma}

To prove Lemma~\ref{lemma:subset}, we first define a weight function $w$ on $[n]\times [n]$ such that the sum of the weight of each line lies within a particular range, as detailed in Lemma~\ref{lemma:weighting}. Using this weight function, we construct the desired subset $S$ of $[n]\times [n]$ by including every square $(i,j)\in [n]\times [n]$ independently with probability $w((i,j))$. The expected degree of each line corresponds to its total weight under $w$, and with high probability, the actual degree will be close to its expectation. We can then apply concentration inequalities and the union bound to show the desired properties.

\begin{lemma}\label{lemma:weighting}
    For all $n\in \mathbb{N}$, there exists a weighting $w:[n]\times [n]\to [17/24,1]$ such that 
    \begin{enumerate}
        \item the sum of the weights in each row and each column is $5n/6\pm 10/3$, \label{lemma:weighting1}
        \item the sum of the weights in each non-reflecting diagonal is less than $59 n/72,$ and \label{lemma:weighting2} 
        \item the sum of the weights in each reflecting diagonal is less than $59n/72.$\label{lemma:weighting3}
    \end{enumerate}
\end{lemma}

\begin{proof}
 
Consider the function $w: [n]\times [n]\to [17/24,1]$ with $$w((i,j)) = \begin{cases}
    43/48 & \text{if } i/(n+1)\in[0,1/3), j/(n+1) \in [0,1/3)\cup (2/3,1] \\
    17/24 & \text{if } i/(n+1)\in[0,1/3), j/(n+1)\in[1/3,2/3]\\
    41/48 & \text{if } i/(n+1)\in[1/3,2/3], j/(n+1) \in [0,1/3)\cup (2/3,1]\\
    19/24 & \text{if } i/(n+1)\in[1/3,2/3], j/(n+1)\in[1/3,2/3]\\
    3/4 & \text{if } i/(n+1)\in(2/3,1], j/(n+1) \in [0,1/3)\cup (2/3,1]\\
    1 & \text{otherwise.}\\
\end{cases}$$ For a line $L\in \cL$, define its weight $w(L)$ to be the sum of the weights of all squares in $L$. The weight function divides the $[n]\times [n]$ grid into 9 \emph{boxes}. We label them $1, \ldots, 9$ from left to right and top to bottom. In Figure~\ref{figure:weighting}, the diagram on the left shows the labeling of the boxes, and the diagram on the right gives the weight of the squares in each box.

\begin{figure}
\begin{minipage}{.5\linewidth}
\centering
    \begin{tikzpicture}
     \draw[step=2] (0,0) grid (6,6);
     \draw (1,5) node[font=\Large]{box 1};
     \draw (3,5) node[font=\Large]{box 2};
     \draw (5,5) node[font=\Large]{box 3};
     \draw (1,3) node[font=\Large]{box 4};
     \draw (3,3) node[font=\Large]{box 5};
     \draw (5,3) node[font=\Large]{box 6};
     \draw (1,1) node[font=\Large]{box 7};
     \draw (3,1) node[font=\Large]{box 8};
     \draw (5,1) node[font=\Large]{box 9};
     \draw (0,6+2/3) -- (6,6+2/3);
     \foreach \i in {0,2/3,4/3,2,10/3,4,16/3,6}
     {\draw (\i,6) -- (\i,20/3);}
     \draw(1/3,19/3) node{1};
     \draw(1,19/3) node{$\ldots$};
     \draw(5/3,19/3) node[font=\Small]{$\frac{n-r}{3}$};
     \draw(17/3,19/3) node{$n$};
     \draw(8/3,19/3) node{$\ldots$};
     \draw(14/3,19/3) node{$\ldots$};
     \draw(11/3,19/3) node[font=\Small]{$\frac{2n+r}{3}$};
\end{tikzpicture}\\
\end{minipage}%
\begin{minipage}{.5\linewidth}
\centering
    \begin{tikzpicture}
    \draw[step=2] (0,0) grid (6,6);
     \draw (1,5) node[font=\huge]{$\frac{43}{48}$};
     \draw (3,5) node[font=\huge]{$\frac{17}{24}$};
     \draw (5,5) node[font=\huge]{$\frac{43}{48}$};
     \draw (1,3) node[font=\huge]{$\frac{41}{48}$};
     \draw (3,3) node[font=\huge]{$\frac{19}{24}$};
     \draw (5,3) node[font=\huge]{$\frac{41}{48}$};
     \draw (1,1) node[font=\huge]{$\frac{3}{4}$};
     \draw (3,1) node[font=\LARGE]{$1$};
     \draw (5,1) node[font=\huge]{$\frac{3}{4}$};
     \draw (0,6+2/3) -- (6,6+2/3);
     \foreach \i in {0,2/3,4/3,2,10/3,4,16/3,6}
     {\draw (\i,6) -- (\i,20/3);}
     \draw(1/3,19/3) node{1};
     \draw(1,19/3) node{$\ldots$};
     \draw(5/3,19/3) node[font=\Small]{$\frac{n-r}{3}$};
     \draw(17/3,19/3) node{$n$};
     \draw(8/3,19/3) node{$\ldots$};
     \draw(14/3,19/3) node{$\ldots$};
     \draw(11/3,19/3) node[font=\Small]{$\frac{2n+r}{3}$};
\end{tikzpicture}
\end{minipage}
\caption{Illustration of the weight function in Lemma~\ref{lemma:weighting}}
\label{figure:weighting}
\end{figure}

Let $r$ be the remainder of $n$ mod 3. Note that boxes $1, 3,7,9$ each have $(n-r)/3$ rows and columns, boxes $4$ and $6$ each have $(n+2r)/3$ rows and $(n-r)/3$ columns, boxes 2 and 8 each have $(n-r)/3$ rows and $(n+2r)/3$ columns, and box 5 has $(n+2r)/3$ rows and columns.

 First, to see (\ref{lemma:weighting1}), observe that $2(43/48)+17/24=2(41/48)+19/24=2(3/4)+1=5/2.$ Since each box has $(n\pm 4)/3$ rows, each row has weight $(5/2)(n\pm4)/3=5n/6\pm 10/3,$ as desired. Similarly, since $43/48+41/48+3/4=17/24+19/24+1=5/2$, each column has weight $(5/2)(n\pm4)/3=5n/6\pm 10/3,$ as desired. 

 Next, we prove (\ref{lemma:weighting2}) by considering the weight of the non-reflecting diagonals. Since the weight function is symmetric, it suffices to consider the non-empty plus diagonals $D_k^+$ for $k\in \{0,1,\ldots, n-1\}$ as $w(D_k^+)=w(D_k^-)$  for all $k\in \{0,1,\ldots, n-1\}$. We claim that the weight of each plus-diagonal is dominated by the weight of $D_0^+$. Indeed, for each $k\in [n]$, the diagonal $D_k^+$ has one fewer square than $D_{k-1}^+$.  Moreover, the number of squares of $D_k^+$ and of $D_{k-1}^+$ in each box differs by at most $1$. Since the weight of each square is between $17/24$ and 1, we have $w(D^+_k)\le w(D^+_{k-1})+2-3(17/24)<w(D^+_{k-1}).$ Hence, $D_0^+$ has the largest weight of the plus-diagonals, as claimed. Note that on $D_0^+,$ there are $(n-r)/3$, $(n+2r)/3$, and $(n-r)/3$ squares in boxes 1, 5, and 9, respectively, so $$w(D_0^+)=\frac{43}{48}\left(\frac{n-r}{3}\right)+\frac{19}{24}\left(\frac{n+2r}{3}\right)+\frac{3}{4}\left(\frac{n-r}{3}\right)=\frac{13}{16}n-\frac{1}{48}r <\frac{59}{72}n.$$ Therefore, the weight of each non-reflecting diagonal is less than $59n/72,$ as desired. 

 Finally, we prove (\ref{lemma:weighting3}) by showing $w(RD_\ell)<59n/72$ for every $\ell\in [n]$. 
 We consider three cases. 

\begin{adjustwidth}{0.5cm}{}
    \textbf{Case 1}: $\ell\le (n-r)/3$ or $\ell\ge (2n+r)/3+1$. 
    
    By the symmetry of the weight function, it suffices to prove the case where $\ell\le (n-r)/3$. 
    We claim that the maximum weight is achieved when $\ell= (n-r)/3$.  Observe that when $\ell\le (n-r)/3$, there are $(n-r)/3-1, \ell, (n+2r)/3-\ell,\ell$, and $(n-r)/3-\ell$ squares of the reflecting diagonal $RD_\ell$ in boxes $1, 2, 5, 6,$ and $9,$ respectively. Comparing the two reflecting diagonals $RD_\ell$ and $RD_{\ell+1}$ for $\ell < (n - r)/3$, we see that $RD_{\ell+1}$ has one more square in each of boxes 2 and 6, and one fewer square in each of boxes 5 and 9. Thus,
    $$w(RD_{\ell+1}) - w(RD_{\ell}) = \frac{17}{24}+\frac{41}{48} - \frac{19}{24} - \frac{36}{48} = \frac{1}{48}.$$ 
    Hence, when $\ell< (n-r)/3$, the weight of $RD_\ell$ increases as $\ell$ increases, and thus is bounded by $$w\left(RD_{(n-r)/3}\right)=\frac{43}{48}\left(\frac{n-r}{3}-1\right)+\frac{17}{24} \left(\frac{n-r}{3}\right)+\frac{3}{4}r+\frac{41}{48}\left(\frac{n-r}{3}\right) = \frac{59}{72}n -\frac{5}{72}r-\frac{43}{48} <\frac{59}{72}n.$$ Therefore, the weight of the reflecting diagonal $RD_\ell$ is less than $59n/72$ when  $\ell\le (n-r)/3$, as desired.
    Since the weight function is symmetric, we have $w(RD_\ell)=w(RD_{n-\ell})$ for all $\ell\in[n]$. Hence, for $\ell\ge (2n+r)/3+1$, the weight of the reflecting diagonal $RD_\ell$ is also less than $59n/72$.
\end{adjustwidth}

\begin{adjustwidth}{0.5cm}{}
    \textbf{Case 2}: $r=2$ and either $\ell=(n-r)/3+1$ or $\ell=(2n+r)/3$. 
    
    For $r=2$ and  $\ell=(n-2)/3+1$, there is one square of the reflecting diagonal $RD_\ell$ in box 5, and $(n-2)/3$ squares in boxes $1,2,$ and $6$, respectively. Hence, we have $$w\left(RD_{(n-2)/3+1}\right)=\left(\frac{43}{48}+\frac{17}{24}+\frac{41}{48}\right)\left(\frac{n-2}{3}\right)+\frac{19}{24} = \frac{59}{72}n-\frac{61}{72}<\frac{59}{72}n.$$ 
    By the symmetry of the weight function, we also have $w(RD_{(2n+2)/3})=w\left(RD_{(n-2)/3+1}\right)<59n/72.$
\end{adjustwidth}

\begin{adjustwidth}{0.5cm}{}

\textbf{Case 3:} $r=2$ and $(n-r)/3+1<\ell<(2n+r)/3$ or $r\in\{0,1\}$ and $(n-r)/3+1\le \ell\le (2n+r)/3$.

In this case, there are $(2n-2r)/3 -\ell+1$, $(n+2r)/3-1$, $\ell-(n+2r)/3,$ $\ell-(n-r)/3-1$, and $(2n+r)/3-\ell$ squares of the reflecting diagonal $RD_\ell$ in boxes $1, 2, 3, 4,$ and $6,$ respectively. Hence, there are $ (n-4r)/3+1$ squares with weight $43/48$, and $(n+2r)/3-1$ squares each with weight $17/24$ and with weight $41/48.$ Thus, the weight $w(RD_\ell)$ for the values of $\ell$ considered in this case is independent of $\ell.$ Thus, in this case, we have $$w(RD_\ell) = \frac{43}{48}\left(\frac{n-4r}{3}+1\right)+\left(\frac{17}{24}+\frac{41}{48}\right)\left(\frac{n+2r}{3}-1\right)=\frac{59}{72}n-\frac{11}{72}r-\frac{2}{3}<\frac{59}{72}n.$$ 
Therefore, the weight of each reflecting diagonal is less than $59n/72$, thereby proving (\ref{lemma:weighting3}).
\end{adjustwidth} We have thus found a weight function $w$ that satisfies all the desired conditions.
\end{proof}

We will use the following standard Chernoff-type bound  (see \cite[Corollary~2.3]{JLR00}) in the proof of Lemma~\ref{lemma:subset}.  

\begin{lemma}[Chernoff Bound]\label{lemma:chernoff}
   If $X$ is the sum of mutually independent Bernoulli random variables with $\mu \coloneqq \Expect{X}$, then for all $\delta \in [0, 1]$, we have
    \begin{equation*}
      \Prob{|X - \mu| \geq \delta \mu} \leq 2e^{-\delta^2 \mu / 3}.
    \end{equation*}
\end{lemma}

We are now equipped with the necessary tools to prove Lemma~\ref{lemma:subset}. 

\begin{proof}[Proof of Lemma~\ref{lemma:subset}]  First, consider choosing $S\subseteq [n]\times [n]$ randomly by including every square $(i,j)$ independently with probability $w((i,j))$, where $w : [n] \times [n] \rightarrow [17,24, 1]$ is the weighting from Lemma~\ref{lemma:weighting}. Note that every line $L \in \cL$ satisfies $\Expect{|L\cap S|} = w(L)$ and $|L \cap S|$ is the sum of $n$ mutually independent Bernoulli random variables.

Let $L\in \cR\cup \cC$. By Lemma~\ref{lemma:weighting}(\ref{lemma:weighting1}), we know that $\Expect{|L \cap S|}=5n/6\pm10/3.$ Applying Lemma~\ref{lemma:chernoff} with $n^{-1/3}$ playing the role of $\delta$ and $|L\cap S|$ playing the role of $X$, we have 
\begin{equation*}
    \Prob{|L \cap S| = (1\pm n^{-1/4})\frac{5n}{6}} \geq 1 - 2\exp\left(-\frac{5n^{1/3}}{19}\right) \geq 1 - \frac{0.01}{2n}.
\end{equation*}
Hence, by a union bound over the $n$ rows and $n$ columns of the chessboard, we see that in $S$, each row and each column contains $(1\pm n^{-1/4})5n/6$ squares with probability at least $0.99$. 

Now, let $L\in \mathcal{ND}\cup \mathcal{RD}$. By Lemma~\ref{lemma:weighting}(\ref{lemma:weighting2}) and (\ref{lemma:weighting3}), we know that  $\Expect{|L \cap S|}<59n/72$, and moreover, by considering increasing the weights, $|L\cap S|$ is stochastically dominated by a random variable $X_L$ which is the sum of $n$ mutually independent Bernoulli random variables with $\Expect{X_L} = 59n/72$. Applying Lemma~\ref{lemma:chernoff} with $59n/72$ playing the role of $\mu,$ $1/118$ playing the role of $\delta$, and $X_L$ playing the role of $X$, we have
\begin{equation*}
    \Prob{|L \cap S| \le  \frac{119n}{144}} \geq \Prob{\left|X_L - \frac{59n}{72}\right| > \frac{n}{144}} \geq 1 - 2\exp\left(-\frac{n}{50976}\right) \geq 1 - \frac{0.01}{3n}.
\end{equation*} By a union bound over the $n$ reflecting diagonals and the $2n$ non-reflecting diagonals of the chessboard, we see that with probability at least 0.99, each reflecting diagonal and each non-reflecting diagonal in $S$ contains at most $119n/144$ squares.

Now, let $A,B\subseteq [n]$ be subsets of size at least $119n/144$. Then $|A\times B| \ge (119n/144)^2$. Since each square of $[n]\times [n]$ is included in $S$ independently with probability at least $17/24$, we have 
\begin{equation*}
    \Expect{|S\cap (A\times B)|} \ge \left(\frac{119n}{144}\right)^2\left(\frac{17}{24}\right)>\frac{29n^2}{60}.
\end{equation*} Applying Lemma~\ref{lemma:chernoff} with $57/58$ playing the role of $\delta$ and  $|S\cap (A\times B)|$ playing the role of $X$, we have 
\begin{equation*}
    \Prob{|S\cap (A\times B)|\ge \frac{n^2}{120}}\ge 1- 2\exp\left(-\frac{n^2}{7}\right)\ge 1 - \frac{0.01}{4^n}.
\end{equation*} By a union bound over at most $4^n$ choices of $A,B\subseteq [n]$, we conclude that $|S\cap (A\times B)|\ge n^2/120$ with probability at least 0.99.
Therefore, with positive probability, there exists a subset $S$ of $[n]\times [n]$ that satisfies all three properties. 
\end{proof}

Now we can finish the proof of our main theorem.

\begin{proof}[Proof of Theorem~\ref{thm:main-thm}] 
Let $n$ be sufficiently large, and consider an $n\times n$ reflecting chessboard represented by $[n]\times [n]$. By Lemma~\ref{lemma:subset}, there exists a subset $S$ of $[n]\times [n]$ such that each row and each column contains $(1\pm n^{-1/4})5n/6$ squares, each non-reflecting diagonal and each reflecting diagonal contains at most $ 119n/144$ squares, and for every $A, B\subseteq [n]$ satisfying $|A|,|B|\geq 199n/144$, we have $|S\cap (A\times B)|\ge n^2/120.$ 

Consider the bipartite graph $G$ with one part $\cR=\bigcup_{i=1}^n\{R_i\}$ corresponding to the $n$ rows of the reflecting chessboard and the other part $\cC=\bigcup_{i=1}^n\{C_i\}$ corresponding to the $n$ columns. For each square $(i,j)\in S$, include $\{R_i,C_j\}$ in the edge set of $G$. Assign to each edge $\{R_i, C_j\}$ in $G$ two colors corresponding to the diagonals containing $(i,j)$ as follows.
\begin{itemize}
    \item If $0\le i+j-(n+1)\le n$, assign $D_{i+j-(n+1)}^+$ to $\{R_i,C_j\}$. Otherwise, assign $RD_{i+j}$ to $\{R_i,C_j\}$.
    \item If $0\le i-j\le n$, assign $D_{i-j}^-$ to $\{R_i,C_j\}$.
     Otherwise, assign $RD_{j-i}$ to $\{R_i,C_j\}$.
\end{itemize} 
Observe that this coloring is 2-bounded, as any two lines of the reflecting chessboard intersect in at most two squares.

We can now apply Lemma~\ref{lemma:rainbow} to $G$ with $\alpha = 1/120, t=2, b=2$, and $d=5n/6$ to find a perfect rainbow matching in $G$. Observe that \begin{enumerate}
    \item every vertex has degree $(1\pm n^{-1/4})5n/6$,
    \item every color has degree at most $(1-1/120)5n/6$, and
    \item there are at least $n^2/120$ edges between any two sets $\cR'\subseteq \cR$ and $\cC'\subseteq \cC$ of size at least $(1-1/120)5n/6$.
\end{enumerate} By Lemma~\ref{lemma:rainbow}, the graph $G$ has a perfect rainbow matching. Therefore, the corresponding subset $S$ of $[n]\times [n]$ contains a reflecting $n$-queens configuration, as desired.
\end{proof}

\bibliographystyle{amsabbrv}
\bibliography{ref}

\end{document}